\documentclass{amsart}

\usepackage{amsthm}
\usepackage[leqno]{amsmath}
\usepackage{latexsym,amsfonts,amssymb}
\usepackage{hyperref}

\newcommand{\numberseries}{\bfseries}   

\newlength{\thmtopspace}                
\newlength{\thmbotspace}                
\newlength{\thmheadspace}               
\newlength{\thmindent}                  

\newtheoremstyle{bfupright head,slanted body}
                {\thmtopspace}{\thmbotspace}
                {\slshape}{\thmindent}{\bfseries}{.}{\thmheadspace}
                {{\numberseries \thmnumber{#2\;}}\thmnote{#3}}

\newtheoremstyle{bfupright head,upright body}
                {\thmtopspace}{\thmbotspace}
                {\upshape}{\thmindent}{\bfseries}{.}{\thmheadspace}
                {{\numberseries \thmnumber{#2\;}}\thmnote{#3}}

\newtheoremstyle{fixed bf head,slanted body}
                {\thmtopspace}{\thmbotspace}{\slshape}
                {\thmindent}{\bfseries}{.}{\thmheadspace}
                {{\numberseries \thmnumber{#2\;}}\thmname{#1}\thmnote{ (#3)}}

\newtheoremstyle{fixed bf head,upright body}
                {\thmtopspace}{\thmbotspace}{\upshape}
                {\thmindent}{\bfseries}{.}{\thmheadspace}
                {{\numberseries \thmnumber{#2\;}}\thmname{#1}\thmnote{ (#3)}}

\newtheoremstyle{numbered paragraph}
                {\thmtopspace}{\thmbotspace}{\upshape}
                {\thmindent}{\upshape}{}{\thmheadspace}
                {{\numberseries \thmnumber{#2.}}}

\newtheoremstyle{unnumbered paragraph}
                {\thmtopspace}{\thmbotspace}{\upshape}
                {\parindent}{\upshape}{}{0pt}

\theoremstyle{bfupright head,slanted body}
\newtheorem{res}{}[section]             \newtheorem*{res*}{}
\theoremstyle{bfupright head,upright body}
               \newtheorem*{bfhpg*}{}
\theoremstyle{fixed bf head,slanted body}
\newtheorem{thm}[res]{Theorem}          \newtheorem*{thm*}{Theorem}
\newtheorem{lem}[res]{Lemma}            \newtheorem*{lem*}{Lemma}

\theoremstyle{fixed bf head,upright body}
       \newtheorem*{dfn*}{Definition}
           \newtheorem*{rmk*}{Remark}
\theoremstyle{numbered paragraph}
\newtheorem{ipg}[res]{}

\newlength{\thmlistleft}        
\newlength{\thmlistright}       
\newlength{\thmlistpartopsep}   
\newlength{\thmlisttopsep}      
\newlength{\thmlistparsep}      
\newlength{\thmlistitemsep}     

\newcounter{eqc} 
\newenvironment{eqc}{\begin{list}{\upshape (\textit{\roman{eqc}})}%
    {\usecounter{eqc}%
      \setlength{\leftmargin}{\thmlistleft}%
      \setlength{\labelwidth}{\thmlistleft}%
      \setlength{\rightmargin}{\thmlistright}%
      \setlength{\partopsep}{\thmlistpartopsep}%
      \setlength{\topsep}{\thmlisttopsep}%
      \setlength{\parsep}{\thmlistparsep}%
      \setlength{\itemsep}{\thmlistitemsep}}}%
  {\end{list}}%

\newcommand{\eqclbl}[1]{{\upshape(\textit{#1})}}

\newcommand{\pgref}[1]{\ref{#1}}

\newcommand{\thmref}[2][Theorem~]{#1\pgref{thm:#2}}

\newcommand{\lemref}[2][Lemma~]{#1\pgref{lem:#2}}
\renewcommand{\eqref}[1]{(\pgref{eq:#1})}
\newcommand{\thmcite}[2][?]{\cite[thm.~#1]{#2}}
\newcommand{\prpcite}[2][?]{\cite[prop.~#1]{#2}}
\newcommand{\lemcite}[2][?]{\cite[lem.~#1]{#2}}
\numberwithin{equation}{res}

\def\urltilda{\kern -.15em\lower .7ex\hbox{\~{}}\kern .04em} 

\newcommand{\setof}[3][\mspace{1mu}]{\{#1#2 \mid #3#1\}}
\newcommand{\ZZ}{\mathbb{Z}}

\newcommand{\dis}{\:\is\:}
\newcommand{\m}{\mathfrak{m}}
\newcommand{\p}{\mathfrak{p}}
\newcommand{\blank}{{\scriptstyle\stackrel{-}{}}}
\newcommand{\is}{\cong}
\renewcommand{\le}{\leqslant}
\renewcommand{\ge}{\geqslant}
\newcommand{\lra}{\longrightarrow}
\newcommand{\mapdef}[4][\rightarrow]{\nobreak{#2\colon #3 #1 #4}}
\newcommand{\Ker}[1]{\nobreak{\operatorname{Ker}#1}}
\renewcommand{\Im}[1]{\nobreak{\operatorname{Im}#1}}
\newcommand{\dif}[2][]{{\partial}^{#2}_{#1}}
\renewcommand{\H}[2][]{\operatorname{H}_{#1}(#2)}
\newcommand{\Shift}[2][]{\mathsf{\Sigma}^{#1}{#2}}
\newcommand{\Tha}[2]{#2_{{\scriptscriptstyle\le}#1}}
\newcommand{\Thb}[2]{#2_{{\scriptscriptstyle\ge}#1}}
\newcommand{\Tsa}[2]{#2_{{\scriptscriptstyle\subset}#1}}
\newcommand{\Tsb}[2]{#2_{{\scriptscriptstyle\supset}#1}}
\newcommand{\SpecR}{\operatorname{Spec}R}
\newcommand{\supp}[2][R]{\operatorname{supp}_{#1}#2}
\newcommand{\cosupp}[2][R]{\operatorname{cosupp}_{#1}#2}
\newcommand{\id}[2][R]{\operatorname{id}_{#1}#2}
\newcommand{\pd}[2][R]{\operatorname{pd}_{#1}#2}
\newcommand{\Hom}[3][R]{\operatorname{Hom}_{#1}(#2,#3)}
\newcommand{\RHom}[3][R]{\operatorname{\mathbf{R}Hom}_{#1}(#2,#3)}
\newcommand{\Ext}[4][R]{\operatorname{Ext}_{#1}^{#2}(#3,#4)}
\newcommand{\tp}[3][R]{\nobreak{#2\otimes_{#1}#3}}
\newcommand{\Ltp}[3][R]{\nobreak{#2\otimes_{#1}^{\mathbf{L}}#3}}
\newcommand{\Cat}[2]{{\mathsf{#2}}(#1)}
\newcommand{\D}[1][R]{\Cat{#1}{D}}

\begin{document}

\title[Injective Modules under Faithfully Flat Ring
Extensions]{Injective Modules under Faithfully Flat\\ Ring Extensions}

\author[Lars Winther Christensen]{Lars Winther Christensen}

\email{lars.w.christensen@ttu.edu}

\urladdr{http://www.math.ttu.edu/\urltilda lchriste}

\author[Fatih K\"{o}ksal]{Fatih K\"{o}ksal}

\email{fatih.koksal@ttu.edu}

\address{Texas Tech University, Lubbock, TX 79409, U.S.A.}

\thanks{Research partly supported by a Simons Foundation Collaboration
  Grant for Mathematicians, award no.\ 281886, and by grant no.\
  H98230-14-0140 from the National Security Agency.}

\date{16 April 2015}



\begin{abstract}
  Let $R$ be a commutative ring and $S$ be an $R$-algebra.  It is
  well-known that if $N$ is an injective $R$-module, then $\Hom{S}{N}$
  is an injective $S$-module. The converse is not true, not even if
  $R$ is a commutative noetherian local ring and $S$ is its
  completion, but it is close: It is a special case of our main
  theorem that in this setting, an $R$-module $N$ with $\Ext{>0}{S}{N}
  =0$ is injective if $\Hom{S}{N}$ is an injective $S$-module.
\end{abstract}

\maketitle

\thispagestyle{empty}

 %
%
%
%
%
%
 %
 %

\section*{Introduction}

\noindent
Faithfully flat ring extensions play a important role in commutative
algebra: Any polynomial ring extension and any completion of a
noetherian local ring is a faithfully flat extension. The topic of
this paper is transfer of homological properties of modules along such
extensions.

In this section, $R$ is a commutative ring and $S$ is a commutative
$R$-algebra.  It is well-known that if $F$ is a flat $R$-module, then
$\tp{S}{F}$ is a flat $S$-module, and the converse is true if $S$ is
faithfully flat over $R$. If $I$ is an injective $R$-module, then
$\tp{S}{I}$ need not be injective over $S$, but it is standard that
$\Hom{S}{I}$ is an injective $S$-module.  Here the converse is not
true, not even if $S$ is faithfully flat over $R$: Let $(R,\m)$ be a
regular local ring with $\m$-adic completion $S \ne R$.  The module
$\Hom{S}{R}$ is then zero---see e.g.\ Aldrich, Enoch, and
Lopez-Ramos~\cite{AEL-02}---and hence an injective $S$-module, but $R$
is not an injective $R$-module, as the assumption $S\ne R$ ensures
that $R$ is not artinian.  In this paper, we get close to a converse
with the following result.

\begin{res*}[Main Theorem]
  Let $R$ be noetherian and $S$ be faithfully flat as an $R$-module;
  assume that every flat $R$-module has finite projective
  dimension. Let $N$ be an $R$-module; if\, $\Hom{S}{N}$ is an
  injective $S$-module and $\Ext{n}{S}{N}=0$ holds for all $n > 0$,
  then $N$ is injective.
\end{res*}

\noindent
The result stated above follows from \thmref{1}. The assumption of
finite projective dimension of flat modules is satisfied by a wide
selection of rings, including rings of finite Krull dimension and
rings of cardinality at most $\aleph_n$ for some natural number $n$;
see Gruson, Jensen et.\ al. \prpcite[6]{CUJ70},
\thmcite[II.(3.2.6)]{LGrMRn71}, and \thmcite[7.10]{LGrCUJ81}. The
projective dimension of a direct sum of modules is the supremum of the
projective dimensions of the summands. A direct sum of flat modules
is flat, so the assumption implies that there is an upper bound $d$
for the projective dimension of a flat module. Notice also that the
condition of $\Ext{n}{S}{N}$ vanishing is finite in the sense that
vanishing is trivial for $n$ greater than the projective dimension of
$S$.
\begin{equation*}
  \ast \ \ast \ \ast
\end{equation*}
The project we report on here is part of K\"oksal's dissertation
work. While the question that started the project---when does
injectivity of $\Hom{S}{N}$ imply injectivity of $N$?---is natural, it
was a result of Christensen and Sather-Wagstaff~\cite{LWCSSW10} that
suggested that a non-trivial answer might be attainable. The main
result in \cite{LWCSSW10} is essentially the equivalent of our Main
Theorem for the relative homological dimension known as
\emph{Gorenstein injective dimension}. That the result was obtained
for the relative dimension before the absolute is already unusual; it
is normally the absolute case that serves as a blueprint for the
relative. In the end, our proof of the Main Theorem bears little
resemblance with the arguments in \cite{LWCSSW10}, and we do not
readily see how to employ our arguments in the setting of that paper.

\section{Injective modules} 

\noindent
In the balance of this paper, $R$ is a commutative noetherian ring and
$S$ is a flat $R$-algebra. By an $S$-module we always mean a left
$S$-module. For convenience, we recall a few basic facts that will be
used throughout without further mention.

\begin{ipg}
  \label{modules}
  A tensor product of flat $R$-modules is a flat $R$-module.  For
  every flat $R$-module $F$ and every injective $R$-module $I$, the
  $R$-module $\Hom{F}{I}$ is injective.

  For every flat $R$-module $F$, the $S$-module $\tp{S}{F}$ is flat,
  and every flat $S$-module is flat as an $R$-module.  For every
  injective $R$-module $I$, the $S$-module $\Hom{S}{I}$ is injective,
  and every injective $S$-module is injective as an $R$-module.

  An $R$-module $C$ is called \emph{cotorsion} if one has
  $\Ext{1}{F}{C}=0$ (equivalently, $\Ext{>0}{F}{C}=0$) for every flat
  $R$-module $F$. It follows by Hom-tensor adjointness that
  $\Hom{F}{C}$ is cotorsion whenever $C$ is cotorsion and $F$ is flat.
\end{ipg}

\begin{ipg}
  Under the sharpened assumption that $S$ is faithfully flat, the exact
  sequence
  \begin{equation}
    \label{eq:seq}
    0 \lra R \lra S \lra S/R \lra 0 
  \end{equation}
  is pure. Another way to say this is that \eqref{seq} is an exact
  sequence of flat $R$-modules; see \cite[Theorems (4.74) and
  (4.85)]{Lam2}.
\end{ipg}

We work mostly in the derived category $\D[R]$ whose objects are
complexes of $R$-modules. The next paragraph fixes the necessary
terminology and notation.

\begin{ipg}
  Complexes are indexed homologically, so that the $i$th differential
  of a complex $M$ is written $\partial^{M}_{i} : M_{i} \rightarrow
  M_{i-1}$.  A complex $M$ is called \textit{bounded above} if
  $M_{v}=0$ holds for all $v \gg 0$, \textit{bounded below} if $M_{v}=0$
  holds for all $v \ll 0$, and $bounded$ if it is bounded above and below.
  Brutal \textit{truncations} of a complex $M$ are denoted
  $\Tha{n}{M}$ and $\Thb{n}{M}$, and good truncations are denoted
  $\Tsa{n}{M}$ and $\Tsb{n}{M}$; cf.~Weibel \cite[1.2.7]{Wei}.

  A complex $M$ is \textit{acyclic} if one has
  $\operatorname{H}(M)=0$, equivalently $M \is 0$ in $\D[R]$. Finally,
  $\RHom{\blank}{\blank}$ denotes the right derived homomorphism
  functor, and $\Ltp{\blank}{\blank}$ denotes the left derived tensor
  product functor.
\end{ipg}

The proof of \thmref{1} passes through a couple of reductions; the
first one is facilitated by the next lemma.

\begin{lem}
  \label{lem:2}
  Let $N$ be an $R$-module of finite injective dimension.  If $S$ is
  faithfully flat, $\Hom{S}{N}$ is an injective $R$-module, and
  $\Ext{n}{S}{N}=0$ holds for all $n > 0$, then $N$ is injective.
\end{lem}

\begin{proof}
  Let $i$ be the injective dimension of $N$.  There exists then an
  $R$-module $T$ such that $\Ext{i}{T}{N} \ne 0$.  Let $E$ be an
  injective envelope of $T$. The exact sequence $ 0 \to T \to E \to X
  \to 0\;$ induces an exact sequence of cohomology modules:
  \begin{equation*}
    \cdots \lra \Ext{i}{E}{N} \lra \Ext{i}{T}{N} 
    \lra \Ext{i+1}{X}{N} \lra \cdots.
  \end{equation*}
  Since $\Ext{i+1}{X}{N}=0$ while $\Ext{i}{T}{N} \neq 0$, we conclude
  that also $\Ext{i}{E}{N}$ is non-zero.  Now apply the functor
  $\tp{\blank}{E}$ to the pure exact sequence \eqref{seq} to get the
  following exact sequence of $R$-modules
  \begin{equation*}
    0 \lra E \lra \tp{S}{E} \lra \tp{S/R}{E} \lra 0\;.
  \end{equation*}
  As $E$ is injective the sequence splits, whence $E$ is a direct
  summand of the module $\tp{S}{E}$.  This implies
  $\Ext{i}{\tp{S}{E}}{N} \neq 0$.  On the other hand, for every $n>0$
  one has
  \begin{equation*}
    \begin{aligned}
      \Ext{n}{\tp{S}{E}}{N}
      & \dis \H[-n]{\RHom{\Ltp{S}{E}}{N}}\\
      & \dis \H[-n]{\RHom{E}{\RHom{S}{N}}}\\
      & \dis \H[-n]{\RHom{E}{\Hom{S}{N}}}\\
      & \dis \Ext{n}{E}{\Hom{S}{N}}\;,
    \end{aligned}
  \end{equation*}
  where the first isomorphism uses that $S$ is flat, the second is
  Hom-tensor adjointness in the derived category, and the third
  follows by the vanishing of $\Ext{>0}{S}{N}$. As $\Hom{S}{N}$ is
  injective, this forces $i=0$; that is, $N$ is injective.
\end{proof}

\begin{ipg}
  \label{BIK}
  Let $\SpecR$ be the set of prime ideals in $R$; for $\p\in\SpecR$
  set \mbox{$\kappa(\p) = R_\p/\p R_\p$}. To an $R$-complex $X$ one
  associates two subsets of $\SpecR$. The (small) \emph{support}, as
  introduced by Foxby \cite{HBF79}, is the set $\supp{X} =
  \setof{\p\in\SpecR}{\H{\Ltp{\kappa(\p)}{X}}\ne 0}$, and the
  \emph{cosupport}, as introduced by Benson, Iyengar and
  Krause~\cite{BIK-12}, is the set $\cosupp{X} =
  \setof{\p\in\SpecR}{\H{\RHom{\kappa(\p)}{X}}\ne 0}$. A complex $X$
  is acyclic if and only if $\supp{X}$ is empty if and only if
  $\cosupp{X}$ is empty; see \cite[(proof of) lem.2.6]{HBF79} and
  \thmcite[4.13]{BIK-12}. The derived category $\D$ is stratified by
  $R$ in the sense of \cite{BIK-11}, see 4.4 \emph{ibid.}, so
  \thmcite[9.5]{BIK-12} yields for $R$-complexes $X$ and $Y$:
  \begin{equation*}
    \cosupp{\RHom{Y}{X}} = \supp{Y} \cap \cosupp{X}\;.
  \end{equation*}
  If $S$ is faithfully flat over $R$ then, evidently, one has
  $\supp{S} = \SpecR$. In this case an $R$-complex $X$ is acyclic if
  $\RHom{S}{X}$ is acyclic.
\end{ipg}

\begin{lem}
  \label{lem:5}
  Let $I$ be an acyclic complex of injective $R$-modules. Assume that
  $S$ is faithfully flat and of finite projective dimension over
  $R$. If\, $\Hom{S}{I}$ is acyclic and $\Hom{S}{\Ker{\dif[n]{I}}}$ is
  an injective $R$-module for every $n \in \ZZ$, then $\Hom{M}{I}$ is
  acyclic for every $R$-module~$M$.
\end{lem}

\begin{proof}
  Let $M$ be an $R$-module; in view of \pgref{BIK} it is sufficient to
  show that the complex $\RHom{S}{\Hom{M}{I}}$ is acyclic. Set $d =
  \pd{S}$ and let $\mapdef{\pi}{P}{S}$ be a projective resolution with
  $P_i =0$ for all $i>d$. To see that the homology
  $\H{\RHom{S}{\Hom{M}{I}}} \is \H{\Hom{P}{\Hom{M}{I}}}$ is zero, note
  first that there is an isomorphism
  \begin{equation*}
    \Hom{P}{\Hom{M}{I}} \dis \Hom{M}{\Hom{P}{I}}\;.
  \end{equation*}
  Fix $m\in \ZZ$; the truncated complex $J = \Tha{m+d+1}{I}$ is a
  bounded above complex of injective $R$-modules, and so is
  $\Hom{P}{J}$. It follows that the induced morphism $\Hom{\pi}{J}$
  is a homotopy equivalence; see \lemcite[10.4.6]{Wei}. This explains
  the first isomorphism in the next display.  The second isomorphism,
  like the equality, is immediate from the definition of Hom. The
  complex $H = \Hom{S}{\Tsa{m+d+1}{I}}$ is acyclic, as $\Hom{S}{I}$ is
  acyclic by assumption and $\Hom{S}{\blank}$ is left exact. By
  assumption $\Hom{S}{\Ker{\dif[m+d+1]{I}}}$ is injective, so $H$ is a
  complex of injective modules; it is also bounded above, so it
  splits. It follows that $\Hom{M}{H}$ is~acyclic.
  \begin{align*}
    \H[m]{\Hom{M}{\Hom{P}{I}}} 
    &= \H[m]{\Hom{M}{\Hom{P}{\Tha{m+d+1}{I}}}} \\
    &\is \H[m]{\Hom{M}{\Hom{S}{\Tha{m+d+1}{I}}}} \\
    &\is \H[m]{\Hom{M}{\Hom{S}{\Tsa{m+d+1}{I}}}} \\
    &= 0\;.\qedhere
  \end{align*}
\end{proof}

\begin{thm}
  \label{thm:1}
  Let $R$ be a commutative noetherian ring over which every flat
  module has finite projective dimension. Let $N$ be an $R$-module and
  $S$ be a faithfully flat $R$-algebra; the following conditions are
  equivalent.
  \begin{eqc}
  \item $N$ is injective.
  \item $\Hom{S}{N}$ is an injective $R$-module and $\Ext{n}{S}{N}=0$
    holds for all $n > 0$.
  \item $\Hom{S}{N}$ is an injective $S$-module and $\Ext{n}{S}{N}=0$
    holds for all $n > 0$.
  \end{eqc}
\end{thm}

\begin{proof}
  It is well-known that \eqclbl{i} implies \eqclbl{iii} implies
  \eqclbl{ii}, so we need to show that \eqclbl{i} follows from
  \eqclbl{ii}. Let $N \to E$ be an injective resolution, then
  $\Hom{S}{E}$ is a complex of injective $R$-modules. By assumption
  $\H[n]{\Hom{S}{E}}$ is zero for $n<0$, so $\Hom{S}{E}$ is an
  injective resolution of the module $\Hom{S}{N}$, which is injective
  by assumption. It follows that the co-syzygies
  \begin{equation*}
    \Ker{\dif[n]{\Hom{S}{E}}} = \Hom{S}{\Ker{\dif[n]{E}}}
  \end{equation*}
  are injective for all $n\le 0$. As remarked in the Introduction,
  there is an upper bound $d$ for the projective dimension of
  a flat $R$-module. Set $K = \Ker{\dif[-d]{E}}$; by \lemref{2} it is
   sufficient to show that $K$ is injective. The complex
  $J = \Shift[d]{(\Tha{-d}{E})}$ is an injective resolution of $K$, so
  we need to show that $\Ext{1}{M}{K} = \H[-1]{\Hom{M}{J}}$ is
  zero for every $R$-module $M$.

  For every flat $R$-module $F$ and all $i > 0$ one has $\Ext{i}{F}{K}
  \is \Ext{i+d}{F}{N} = 0$ by dimension shifting; that is, $K$ is
  cotorsion. For every $i>0$ the $i$-fold tensor product
  $(S/R)^{\otimes i}$ is a flat $R$-module, and we set $(S/R)^{\otimes
    0} = R$. Let $\eta$ denote the pure exact sequence \eqref{seq};
  splicing together the exact sequences of flat modules
  $\tp{\eta}{(S/R)^{\otimes i}}$ for $i\ge 0$ one gets an acyclic
  complex
  \begin{equation*}
    G = 0 \to R \to S \to \tp{S}{S/R} \to \tp{S}{(S/R)^{\otimes 2}} \to 
    \cdots \to \tp{S}{(S/R)^{\otimes i}} \to \cdots
  \end{equation*}
  concentrated in non-positive degrees. As $K$ is cotorsion, the
  functor $\Hom{\blank}{K}$ leaves each sequence $\tp{\eta}{(S/R)^{\otimes
      i}}$ exact, so the complex $\Hom{G}{K}$ is acyclic. For every
  $n>0$, the $R$-module
  \begin{equation*}
    \Hom{G}{K}_n = \Hom{\tp{S}{(S/R)^{\otimes n-1}}}{K} \is 
    \Hom{(S/R)^{\otimes n-1}}{\Hom{S}{K}} 
  \end{equation*}
  is injective; indeed, $\Hom{S}{K}$ is injective and $(S/R)^{\otimes
    n-1}$ is flat. Moreover, one has $\Hom{G}{K}_0 \is K$, so the
  complexes $\Thb{1}{\Hom{G}{K}}$ and $J$ splice together to yield an
  acyclic complex $I$ of injective $R$-modules.

  We argue that \lemref{5} applies to $I$. For $n < 0$ one has
  $\H[n]{\Hom{S}{I}} = \H[n]{\Hom{S}{J}} = \H[n-d]{\Hom{S}{E}} = 0$,
  and the module $\Hom{S}{\Ker{\dif[n]{I}}} =
  \Hom{S}{\Ker{\dif[n-d]{E}}}$ is injective. For $n \ge 0$ one has
  \begin{equation*}
    \Ker{\dif[n]{I}} = \Hom{\Im{\dif[-n]{G}}}{K} = \Hom{(S/R)^{\otimes n}}{K}\;.
  \end{equation*}
  Since $K$ is cotorsion and $(S/R)^{\otimes n}$ is flat, the module
  $\Ker{\dif[n]{I}}$ is cotorsion. The truncated complex $I_{\le n+1}$
  is an injective resolution of the module $\Ker{\dif[n+1]{I}}$, so
  for all $n \ge 0$ one has $\H[n]{\Hom{S}{I}} =
  \Ext{1}{S}{\Ker{\dif[n+1]{I}}} = 0$. Furthermore, the $R$-module
  $\Hom{S}{\Ker{\dif[n]{I}}} \is \Hom{(S/R)^{\otimes n}}{\Hom{S}{K}}$
  is injective.

  Now it follows from \lemref{5} that $\Hom{M}{I}$ is acyclic for
  every $R$-module $M$; in particular, one has $\H[-1]{\Hom{M}{J}} =
  \H[-1]{\Hom{M}{I}} = 0$.
\end{proof}

\section{Injective dimension}

\noindent
To draw the immediate consequences of our theorem, we need some
terminology.

\begin{ipg}
  An $R$-complex $I$ is \textit{semi-injective} if it is a complex of
  injective $R$-modules and the functor $\Hom{\blank}{I}$ preserves
  acyclicity.  A \emph{semi-injective resolution} of an $R$-complex
  $N$ is a semi-injective complex $I$ that is isomorphic to $N$ in
  $\D[R]$. If $N$ is a module, then an injective resolution of $N$ is
  a semi-injective resolution in this sense. The \textit{injective
    dimension} of an $R$-complex $N$ is denoted $\id{N}$ and defined~as
  \begin{equation*}
\id{N} = \inf \left\{ i \in \mathbb{Z}
  \left| \begin{array}{rl}
      \mbox{There is a semi-injective resolution} \\
      \mbox{$I$ of $N$ with $I_{n}=0$ for all $n < -i$}
    \end{array} \right. 
\right\}\;; 
\end{equation*}
see \cite[2.4.I]{LLAHBF91}, where ``DG-injective'' is the same as ``semi-injective''.
\end{ipg}

\begin{thm}
  \label{thm:2}
  Let $R$ be a commutative noetherian ring over which every flat
  module has finite projective dimension, and let $S$ be a flat
  $R$-algebra. For every $R$-complex $N$ there are inequalities
  \begin{equation*}
    \id{N} \ge \id[S]{\RHom{S}{N}} \ge \id{\RHom{S}{N}}\;,
  \end{equation*}
  and equalities hold if $S$ is faithfully flat.
\end{thm}

\begin{proof}
  Let $N$ be an $R$-complex and let $I$ be a semi-injective resolution
  of $N$.  In $\D[S]$ there is an isomorphism $\RHom{S}{N} \dis
  \Hom{S}{I}$. It follows by Hom-tensor adjointness that $\Hom{S}{I}$
  is a semi-injective $S$-complex, whence the left-hand inequality
  holds. As $S$ is flat over $R$, Hom-tensor adjointness also shows
  that every semi-injective $S$-complex is semi-injective over $R$. In
  particular, any semi-injective resolution of $\RHom{S}{N}$ over $S$
  is a semi-injective resolution over~$R$, and the second inequality
  follows.

  Assume now that $S$ is faithfully flat and that $\id{\RHom{S}{N}}
  \le i$ holds for some integer $i$.  Let $I$ be a semi-injective
  resolution of $N$; our first step is to prove that the $R$-module $K
  = \Ker{\partial^{I}_{-i}}$ is injective. As $\Hom{S}{\blank}$ is
  left exact one has
  \begin{equation*}
    \Ker{\partial^{\Hom{S}{I}}_{-i}} \dis \Hom{S}{K}\;.
  \end{equation*} 
  In $\D$ there is an isomorphism $\Hom{S}{I} \is \RHom{S}{N}$, and by
  previous arguments the $R$-complex $\Hom{S}{I}$ is semi-injective.
  It now follows from \cite[2.4.I]{LLAHBF91} that the $R$-module
  $\Hom{S}{K}$ is injective, and the truncated complex
  $\Tsb{-i}{\Hom{S}{I}} = \Hom{S}{\Tsb{-i}{I}}$ is isomorphic to
  $\RHom{S}{N}$ in $\D$. In particular, one has
  \begin{equation*}
    \Ext{n}{S}{K} = \H[-n]{\Hom{S}{\Shift[i]{(\Tha{-i}{I})}}} 
    = \H[-i-n]{\RHom{S}{N}} = 0
  \end{equation*}
  for all $n > 0$, so $K$ is injective by \thmref{1}.

  To conclude that $N$ has injective dimension at most $i$, it is now
  sufficient to show that $\H[n]{N}=0$ holds for all $n < -i$;
  see \cite[2.4.I]{LLAHBF91}. Let $X$ be the cokernel of the embedding
  $\mapdef{\iota}{\Tsb{-i}{I}}{I}$; the sequence $0 \to \Tsb{-i}{I}
  \to I \to X \to 0$ is a degree-wise split exact sequence of complexes
  of injective modules. In the induced exact sequence
  \begin{equation*}
    0 \lra \Hom{S}{\Tsb{-d}{I}} \lra \Hom{S}{I} \lra \Hom{S}{X} \lra 0\;,
  \end{equation*}
  the embedding is a homology isomorphism, so $\Hom{S}{X}$ is acyclic. As
  $X$ is a bounded above complex of injective modules, it is
  semi-injective. That is, the complex $\RHom{S}{X}$ is acyclic, and then it
  follows that $X$ is acyclic; see~\pgref{BIK}. Thus $\iota$ is a
  quasi-isomorphism, whence one has $\H[n]{N} = \H[n]{I} = 0$ for all
  $n < -i$.
\end{proof}

\bibliographystyle{amsplain} 

\def\cprime{$'$}
  \providecommand{\arxiv}[2][AC]{\mbox{\href{http://arxiv.org/abs/#2}{\sf
  arXiv:#2 [math.#1]}}}
  \providecommand{\oldarxiv}[2][AC]{\mbox{\href{http://arxiv.org/abs/math/#2}{\sf
  arXiv:math/#2
  [math.#1]}}}\providecommand{\MR}[1]{\mbox{\href{http://www.ams.org/mathscinet-getitem?mr=#1}{#1}}}
  \renewcommand{\MR}[1]{\mbox{\href{http://www.ams.org/mathscinet-getitem?mr=#1}{#1}}}
\providecommand{\bysame}{\leavevmode\hbox to3em{\hrulefill}\thinspace}
\providecommand{\MR}{\relax\ifhmode\unskip\space\fi MR }
\providecommand{\MRhref}[2]{%
  \href{http://www.ams.org/mathscinet-getitem?mr=#1}{#2}
}
\providecommand{\href}[2]{#2}

\end{document}